\documentclass{amsart}

\usepackage{epic, eepic}

\newtheorem{theorem}{Theorem}[section]
\newtheorem{proposition}[theorem]{Proposition}
\newtheorem{lemma}[theorem]{Lemma}
\newtheorem{claim}{Claim}

\begin{document}

\title[Surface bundles with rank two fundamental groups] 
{Constructions of surface bundles\\ with rank two fundamental groups} 

\author{Kazuhiro Ichihara}
\address{Department of Mathematics, College of Humanities and Sciences, Nihon University, 
3-25-40 Sakurajosui, Setagaya-ku, Tokyo 156-8550, Japan.}
\email{ichihara@math.chs.nihon-u.ac.jp}

\author{Mitsuhiko Takasawa}
\address{Department of Mathematical and Computing Sciences, 
Tokyo Institute of Technology, 
Ookayama 2--12--1, Meguro-ku, Tokyo 152--8552, Japan} 
\email{takasawa@is.titech.ac.jp}

\thanks{The first author is partially supported by
Grant-in-Aid for Young Scientists (B), No.~20740039,
Ministry of Education, Culture, Sports, Science and Technology, Japan.}

\date{\today}

\begin{abstract}
We give a construction 
of hyperbolic $3$-manifolds with rank two fundamental groups 
and report an experimental search to find such manifolds. 
Our manifolds are all surface bundles over the circle 
with genus two surface fiber. 
For the manifolds so obtained, 
we then examine whether they are of Heegaard genus two or not. 
As a byproduct, we give 
an infinite family of knots in the $3$-sphere 
whose knot groups are of rank two. 
\end{abstract}
\maketitle

%=======================================================
\section{Introduction}

The \textit{Heegaard genus} and the \textit{rank} of the fundamental group 
are well-known and well-studied complexities of $3$-manifolds 
(see the next section for precise definitions). 

As an extension of the famous Poincar\'{e} conjecture, 
Waldhausen had asked in \cite{Wa} 
whether 
the Heegaard genus of a compact orientable $3$-manifold $M$ 
is equal to 
the rank of its fundamental group $\pi_1(M)$ or not. 
The Poincar\'{e} conjecture states that 
it would be answered affirmatively in the simplest case; 
when $M$ is closed, the rank of $\pi_1(M)$ is zero 
if and only if the Heegaard genus of $M$ is zero. 

In general, the question was negatively answered in \cite{BoZ}. 
In fact, they gave a family of 
closed $3$-manifolds of Heegaard genus three 
with rank two fundamental groups. 
Those $3$-manifolds are Seifert fibered spaces, 
and the family was extended to more wider one in \cite{MS}. 
Also such examples of graph manifolds were obtained in \cite{W}. 
However, as far as the authors know, 
no such examples are known for either hyperbolic $3$-manifolds 
or $3$-manifolds with non-empty boundary \cite[Problem 3.92]{K}. 

In this paper, 
we give systematical and experimental constructions 
of hyperbolic $3$-manifolds with rank two fundamental groups, 
and examine whether they are of Heegaard genus two or not. 

First, we will give a construction of 
surface bundles over the circle 
with genus two fiber 
each of which has the fundamental group of rank two. 
With some exceptions, they will be shown to have 
pseudo-Anosov monodromies, and so, 
they admit hyperbolic structure 
by \cite{Th} (see \cite{O} for a detailed proof). 
Actually, for these manifolds, we will verify that 
all of them are of Heegaard genus two. 

We here remark that, in \cite{Johnson}, 
a characterization of the monodromy maps 
for closed orientable surface bundles 
with genus two Heegaard splittings is obtained. 

As a byproduct, our construction yields 
an infinite family of knots in the $3$-sphere $S^3$ 
whose knot groups are of rank two. 
Precisely we will give an infinite family of 
hyperbolic, genus two, fibered knots in $S^3$ with rank two knot groups. 
Also they all can be checked to have Heegaard genus two; 
that is, they are tunnel number one knots. 
Thus they give supporting evidence 
for the conjecture raised in \cite{Scha1}: 
A knot group has two generators and one relater if and only if 
the knot is tunnel number one. 
See \cite[Problem 1.10]{K}, \cite[Problem 1.73]{K} for related problems 
and \cite{Bl} for a partial solution. 
There are also known examples of knots in $S^3$ 
which are tunnel number one, genus two, fibered knots. 
They are given by two-bridge knots, i.e., knots with bridge index $2$. 
Here a \textit{bridge index} of a knot in $S^3$ is defined as 
the minimal number of local maxima (or local minima) up to ambient isotopy. 
Precisely, the knots are 
$5_1$, $6_2$, $6_3$, $7_6$, $7_7$, $8_{12}$, in the knot table \cite{R}, 
and also are composite knots coming from trefoil and figure-eight knots. 
Actually it is easy to see that two-bridge knots are tunnel number one by definition, 
and it is shown from \cite[Proposition 2]{GK} 
that these are only fibered knots of genus two among two-bridge knots. 
Moreover, in \cite[Corollary 5.4]{jong}, Jong showed that 
they are only genus two fibered knots among alternating knots. 

Second, we report on computer experiments 
for finding examples of surface bundles 
with rank two fundamental groups. 
In fact, we did generate over 100,000 hyperbolic surface bundles 
with rank two fundamental groups by computer, 
however we could not find examples of Heegaard genus more than two.

%=======================================================
\section*{Acknowledgments}
The authors would like to thank Yo'av Rieck for 
giving them an idea to construct surface bundles 
with rank two fundamental groups. 
Also they thanks to 
Mikami Hirasawa, Akio Kawauchi and Takuji Nakamura for 
suggestions about Proposition \ref{prop-knots}.

%=======================================================
\section{Preliminary}

Throughout the paper, unless otherwise stated, 
all manifolds are assumed to be connected and orientable. 
In this section, we use $F$ to denote a closed surface 
and $M$ a compact $3$-manifold with boundary $\partial M$. 

\subsection{}
A \textit{compression body} $C$ is defined as 
a compact $3$-manifold obtained from the product $F \times [0,1]$ by 
attaching 2-handles on $F \times \{ 1 \}$ and then 
capping off the resulting 2-sphere boundary components by $3$-handles. 
The subsurface of the boundary $\partial C$ 
corresponding to $F \times \{ 0 \}$ is denoted by $\partial_+ C$, and 
then the residual set $\partial C - \partial_+ C$ is denoted by $\partial_- C$. 
Under this setting, a \textit{handlebody} is defined as 
a compression body $C$ with $\partial_- C = \emptyset$. 

By a \textit{Heegaard surface} in $M$, 
we mean an embedded surface $S$ in $M$ 
which separates $M$ into a handlebody $C_1$ and a compression body $C_2$ 
with $\partial M = \partial_- C_2$ and 
$S = \partial_+ C_1 = \partial_+ C_2$. 
Such a splitting of $M$ is called a \textit{Heegaard splitting}. 
It is well-known that 
every compact $3$-manifold admits a Heegaard splitting. 
Thus one can consider the minimal genus of Heegaard surfaces in $M$, 
which is called the \textit{Heegaard genus} of $M$. 

If $M$ admits a Heegaard surface of genus $g$, 
then, by the Van-Kampen's theorem, 
its fundamental group $\pi_1(M)$ admits 
a presentation with $g$ generators. 
Since $M$ admits at least one Heegaard splitting, 
$\pi_1 (M)$ is always finitely generated. 
Among presentations of $\pi_1(M)$, 
one can consider the minimal number of generators, 
which we call the \textit{rank} of the group $\pi_1 (M)$. 
It then follows that 
the rank of $\pi_1(M)$ is less than or equal to 
the genus of any Heegaard surface in $M$. 

As we stated in the previous section, 
we consider the question 
whether the rank of $\pi_1(M)$ is equal to the Heegaard genus of $M$ or not. 
By the observation above, this question is equivalent to that 
whether a presentation of $\pi_1(M)$ with minimal number of generators 
is induced from a Heegaard splitting of $M$. 

\subsection{}
By a \textit{surface bundle}, 
we mean a $3$-manifold which fibers over the circle. 
Any surface bundle with fiber $F$ can be regarded as 
$$ (F \times [0,1] ) / \{ (x,1) = (f(x), 0) \}_{x \in F} $$
with some orientation preserving homeomorphism $f$ of $F$. 
This $f$ is called the \textit{monodromy} of the surface bundle. 
Throughout the paper, 
the surface bundle with monodromy $f$ is denoted by $M_f$. 

Note that if $f$ and $f'$ are isotopic, then 
$M_f$ and $M_{f'}$ are homeomorphic. 
Thus for an element $[f]$ of the mapping class group of $F$, 
the surface bundle $M_f$ is uniquely determined. 
Also note that for any $f$ and its conjugate $gfg^{-1}$ by some $g$, 
$M_f$ and $M_{gfg^{-1}}$ are homeomorphic. 
In fact, it was shown in \cite{Neu} that 
when the rank of the first homology is one, 
a surface bundle $M_f$ is homeomorphic to another $M_{f'}$ 
if and only if 
$f'$ is isotopic to a conjugate of $f$. 

The question which we consider 
was answered affirmatively for torus bundles in \cite{TO}: 
The Heegaard genus of a torus bundle over the circle 
is equal to the rank of the fundamental group. 

\subsection{}
In the following, 
let $F_2$ denote a closed orientable surface of genus two. 
Let $C_1, C_2, \dots, C_5$ 
be the simple closed curves on $F_2$ 
depicted in Figure \ref{fig1}. 
For each $i$, $1 \leq i \leq 5$, 
let $D_i$ denote the Dehn twist along $C_i$. 

\unitlength=0.6mm
\begin{figure}[ht]
 \begin{picture}(100,50)
		\newsavebox{\genus}	% genus
		\savebox{\genus}(12,4) 
		{ 
		\begin{picture}(12,4)
			\qbezier (-6,2)(0,6)(6,2)
			\qbezier (-6,2)(0,-2)(6,2)
		\end{picture} 
		} 
		\newsavebox{\longitude}	% C_2 & C_4
		\savebox{\longitude}(24,18) 
		{ 
		\begin{picture}(24,18)
			\qbezier (0,8)(0,0)(8,0)	
			\qbezier (0,10)(0,18)(8,18)
			\qbezier (16,18)(24,18)(24,10)
			\qbezier (16,0)(24,0)(24,8)
			\multiput(0,8)(24,0){2}{\line(0,1){2}}
			\multiput(8,0)(0,18){2}{\line(1,0){8}}
		\end{picture} 
		} 
  \qbezier (20,5)(10,5)(10,15)		%surf.
  \qbezier (10,35)(10,45)(20,45)
  \qbezier (80,45)(90,45)(90,35)
  \qbezier (80,5)(90,5)(90,15)
  \multiput(10,15)(80,0){2}{\line(0,1){20}}
  \multiput(20,5)(0,40){2}{\line(1,0){60}}
  \put  (33,23) {\usebox{\genus} }	%genus
  \put  (67,23) {\usebox{\genus} }
  \qbezier (33,23)(38,14)(33,5)		%C_1
  \qbezier [12](33,23)(28,14)(33,5)
  \multiput(20,16)(36,0){2}{\usebox{\longitude}}%C_2, C_4
  \qbezier (39,25)(50,30)(61,25)		%C_3
  \qbezier [12](39,25)(50,20)(61,25)
  \qbezier (67,23)(72,14)(67,5)		%C_5
  \qbezier [12](67,23)(62,14)(67,5)
  \put (31,-2){ $C_1$}
  \put (16,34){ $C_2$}
  \put (45.5,29){ $C_3$}
  \put (75,33){ $C_4$}
  \put (65,-2){ $C_5$}
 \end{picture}
 \caption{}\label{fig1}
\end{figure}

By \cite{Li}, every orientation preserving homeomorphism of $F_2$ 
is isotopic to a product of 
a finite number of the Dehn twists $D_1, D_2, D_3, D_4, D_5$ and their inverses. 

%=======================================================
\section{Construction of surface bundle}

Our first theorem is the following. 

\begin{theorem}\label{thm1}
Let $n$ be an arbitrary integer and 
$\boldsymbol{\varepsilon}= (\varepsilon_1, \varepsilon_2, \varepsilon_3, \varepsilon_4)$ 
a quadruple of integers $\varepsilon_i$ 
each of which is either $+1$ or $-1$. 
Let $M_{\boldsymbol{\varepsilon},n}$ be the surface bundle over the circle 
with monodromy 
$$f_{\boldsymbol{\varepsilon},n} = {D_2}^{\varepsilon_2} \circ 
{D_1}^{\varepsilon_1} \circ {D_3}^{\varepsilon_3} \circ 
{D_4}^{\varepsilon_4} \circ {D_5}^n .$$
Then the rank of $\pi_1( M_{\boldsymbol{\varepsilon},n} )$ is always two and 
the Heegaard genus of $M_{\boldsymbol{\varepsilon},n}$ is also two. 
\end{theorem}

Of course, the second assertion implies the first assertion, 
but we will directly prove the first assertion 
without assuming the second assertion. 

In fact, for the case of $n=0$, 
the manifold $M_{\boldsymbol{\varepsilon},0}$ 
is that obtained from a $2$-bridge knot in $S^3$ by $0$-surgery. 
Thus it naturally admits a Heegaard surface of genus two and 
so the theorem follows immediately  in this case. 
Our manifolds can be regarded as 
an extension of the class of such manifolds. 

The first assertion of the theorem follows from the next lemma. 

\begin{lemma}\label{lem1}
With the same notations as in Theorem $\ref{thm1}$, 
the rank of $\pi_1( M_{\boldsymbol{\varepsilon},n} )$ is always two. 
\end{lemma}

\begin{proof}
Consider the oriented loops $a_1, a_2, a_3, a_4$ on $F_2$ 
illustrated in Figure \ref{fig2}. 

\unitlength=0.6mm
\begin{figure}[hbt]
 \begin{picture}(100,50)
		\newsavebox{\gen}	% genus
		\savebox{\gen}(12,4) 
		{ 
		\begin{picture}(12,4)
			\qbezier (-6,2)(0,6)(6,2)
			\qbezier (-6,2)(0,-2)(6,2)
		\end{picture} 
		} 
  \qbezier (20,5)(10,5)(10,15)		%surf.
  \qbezier (10,35)(10,45)(20,45)
  \qbezier (80,45)(90,45)(90,35)
  \qbezier (80,5)(90,5)(90,15)
  \multiput(10,15)(80,0){2}{\line(0,1){20}}
  \multiput(20,5)(0,40){2}{\line(1,0){60}}
  \put  (33,23) {\usebox{\gen}}	% genus
  \put  (67,23) {\usebox{\gen}}
  \put  (51.3,15) {\circle* {2.5}}		% base point
  \put (34,23){ \vector(2,-1){7} }	% a_1
  \put (41,19.5){ \line(2,-1){8} }
  \qbezier [10](35,5.5)(31.5,14)(35,23)
  \qbezier (35,5)(43.5,10)(52,15)
  \put (33.5,1){$a_1$}
  \put (50.5, 16){\vector(-1,4){3}}	% a_2
  \qbezier (47.5,28)(43.5,38)(32,37.5)
  \put (32,37.5){\line(-1,0){3}}
  \qbezier (29,37.5)(20,37.5)(20,28.5)
  \put (20,28.5){\line(0,-1){7}}
  \qbezier (20,21.5)(20,12.5)(29,12.5)
  \qbezier (29,12.5)(40,13)(50,15)
  \put (18,37.5){$a_2$}
  \put (65.5,23){ \vector(-2,-1){7} }	% a_3
  \put (58.5,19.5){ \line(-2,-1){8} }
  \qbezier [10](67,5.05)(70,14)(67,23.02)
  \qbezier (67,5)(58,10)(51,15)
  \put (66,1){$a_3$}
  \put (51.5, 16){\vector(1,4){3}}	% a_4
  \qbezier (54.5,28)(58.5,38)(69,37.5)
  \put (69,37.5){\line(1,0){2}}
  \qbezier (71,37.5)(80,37.5)(80,28.5)
  \put (80,28.5){\line(0,-1){7}}
  \qbezier (80,21.5)(80,12.5)(71,12.5)
  \qbezier (71,12.5)(60,13)(52,15)
  \put (78,37){$a_4$}
  \end{picture}
 \caption{}\label{fig2}
\end{figure}

We abuse the notations $a_1, a_2, a_3, a_4$ to denote 
the elements of the fundamental group $\pi_1 (F_2)$ 
represented by the corresponding loops. 

Let $\Phi_{\boldsymbol{\varepsilon},n}$ be the isomorphism of $\pi_1 (F_2)$ 
induced by $f_{\boldsymbol{\varepsilon},n}$. 
Then it is well-known that $\pi_1( M_{\boldsymbol{\varepsilon},n} )$ decomposes as 
a semidirect product of $\pi_1 (F_2)$ and $\mathbb{Z}$. 
Precisely, 
$\pi_1( M_{\boldsymbol{\varepsilon},n} )$ has the following presentation: 
\begin{equation}\label{prn1}
\Big\langle 
a_1, a_2, a_3, a_4, t \ \Big\vert \ t^{-1} a_i t = \Phi_{\boldsymbol{\varepsilon},n} (a_i), 
\ (1 \leq i \leq 4), \ [a_1, a_2] = [a_3, a_4] 
\Big\rangle ,
\end{equation}
where $t$ denotes a generator of the infinite cyclic factor. 

\begin{claim}
For each $i \in \{ 1,2,3 \}$, 
the image of $a_i$ under $\Phi_{\boldsymbol{\varepsilon},n}$ is represented by a word 
which contains just one $a_{i+1}$ (or $(a_{i+1})^{-1}$) and no letters $a_j$ $(j > i+1)$. 
\end{claim}

\begin{proof}
We first describe 
the isomorphism $\Phi_{\boldsymbol{\varepsilon},n}: \pi_1 (F_2) \to \pi_1 (F_2)$ in detail. 
Let $\Delta_i$ denote the isomorphism of $\pi_1 (F_2)$ 
induced by the Dehn twists $D_i$ for $1 \leq i \leq 4$. 
Then 
$$\Phi_{\boldsymbol{\varepsilon},n} = 
{\Delta_2}^{\varepsilon_2} \circ 
{\Delta_1}^{\varepsilon_1} \circ {\Delta_3}^{\varepsilon_3} \circ 
{\Delta_4}^{\varepsilon_4} \circ {\Delta_5}^n $$
holds. 

By Figure \ref{fig1} and \ref{fig2}, 
the actions of $\Delta_i$'s on $a_j$'s are described as follows. 
$$
\Big( {\Delta_i}^{\varepsilon}(a_j) \Big)_{1 \leq i \leq 5, 1 \leq j \leq 4}
= 
\left(
\begin{array}{cccc}
a_1 &	a_2 (a_1)^{\varepsilon}&	a_3 &		a_4 \\
a_1 (a_2)^{-\varepsilon}& 	a_2 &	a_3 &	a_4 \\
a_1 &	((a_3)^{-1} a_1)^{\varepsilon} a_2 &	a_3 &	
((a_3)^{-1} a_1)^{\varepsilon} a_4 \\
a_1 &		a_2 &		a_3 (a_4)^{\varepsilon} &	a_4 \\
a_1 &		a_2 &		a_3 &		a_4 (a_3)^{-\varepsilon} \\
\end{array}
\right),
$$
where $\varepsilon=\pm 1$. 

For $a_1$, this implies that 
$$
\Phi_{\boldsymbol{\varepsilon},n} (a_1) = 
{\Delta_2}^{\varepsilon_2} \circ 
{\Delta_1}^{\varepsilon_1} \circ {\Delta_3}^{\varepsilon_3} \circ 
{\Delta_4}^{\varepsilon_4} \circ {\Delta_5}^n (a_1) 
=  {\Delta_2}^{\varepsilon_2} (a_1) = 
a_1 (a_2)^{-\varepsilon_2}. 
$$
Thus the claim holds for the case that $i=1$. 

For $a_2$, we have 
\begin{eqnarray*}
\Phi_{\boldsymbol{\varepsilon},n} (a_2) 
&=&
{\Delta_2}^{\varepsilon_2} \circ 
{\Delta_1}^{\varepsilon_1} \circ {\Delta_3}^{\varepsilon_3} \circ 
{\Delta_4}^{\varepsilon_4} \circ {\Delta_5}^n (a_2) 
= {\Delta_2}^{\varepsilon_2} \circ {\Delta_1}^{\varepsilon_1} \circ {\Delta_3}^{\varepsilon_3} (a_2) \\
&=& {\Delta_2}^{\varepsilon_2} \circ {\Delta_1}^{\varepsilon_1} ( ((a_3)^{-1} a_1)^{\varepsilon_3} a_2 ). 
\end{eqnarray*}
From the matrix above, note that $a_3$ is invariant 
under the isomorphisms $(\Delta_1)^{\pm 1}$, $(\Delta_2)^{\pm 1}$ and 
the image of $a_1$, $a_2$ under $(\Delta_1)^{\pm 1}$, $(\Delta_2)^{\pm 1}$ are 
represented by the word with letters $a_1$, $a_2$. 
Thus $\Phi_{\boldsymbol{\varepsilon},n} (a_2) 
={\Delta_2}^{\varepsilon_2} \circ {\Delta_1}^{\varepsilon_1} 
( ((a_3)^{-1} a_1)^{\varepsilon} a_2 )$ is represented by the word 
with a single letter $a_3$ and the letters $a_1$, $a_2$. 

Finally for $a_3$, we have 
\begin{eqnarray*}
\Phi_{\boldsymbol{\varepsilon},n} (a_3) 
&=&
{\Delta_2}^{\varepsilon_2} \circ 
{\Delta_1}^{\varepsilon_1} \circ {\Delta_3}^{\varepsilon_3} \circ 
{\Delta_4}^{\varepsilon_4} \circ {\Delta_5}^n (a_3) \\
&=& {\Delta_2}^{\varepsilon_2} \circ {\Delta_1}^{\varepsilon_1} \circ 
{\Delta_3}^{\varepsilon_3} \circ {\Delta_4}^{\varepsilon_4} (a_3) \\
&=& {\Delta_2}^{\varepsilon_2} \circ {\Delta_1}^{\varepsilon_1} 
( a_3 (((a_3)^{-1} a_1)^{\varepsilon_3} a_4 )^{\varepsilon_4} ). 
\end{eqnarray*}
Again, from the matrix above, note that $a_4$ is invariant 
under the isomorphisms $(\Delta_1)^{\pm 1}$, $(\Delta_2)^{\pm 1}$ and 
the image of $a_1$, $a_2$, $a_3$ 
under $(\Delta_1)^{\pm 1}$, $(\Delta_2)^{\pm 1}$ are 
represented by the word with letters $a_1$, $a_2$, $a_3$. 
Thus 
$$\Phi_{\boldsymbol{\varepsilon},n} (a_3)
={\Delta_2}^{\varepsilon_2} \circ {\Delta_1}^{\varepsilon_1} 
( a_3 (((a_3)^{-1} a_1)^{\varepsilon_3} a_4 )^{\varepsilon_4} )$$ 
is represented by the word 
with single letter $a_4$ and the letters $a_1$ , $a_2$ and $a_3$. 
\end{proof}

The claim enables us to reduce 
the number of the generators of the presentation (\ref{prn1}), 
and we see that $\pi_1( M_{\boldsymbol{\varepsilon},n} )$ is generated 
by $a_1$ and $t$ only. 
This completes the proof of the lemma. 
\end{proof}

Next we show that the Heegaard genus of $M_{\boldsymbol{\varepsilon},n}$ is also two 
independently from Lemma \ref{lem1}. 

\begin{lemma}\label{lem2}
With the same notations as in Theorem $\ref{thm1}$, 
every $M_{\boldsymbol{\varepsilon},n}$ is of Heegaard genus two. 
\end{lemma}

\begin{proof}
To prove the lemma, 
we first create a surgery description of $M_{\boldsymbol{\varepsilon},n}$. 

Let $L$ be the link defined as follows. 
Prepare five copies of annulus embedded in $S^3$ and plumb them. 
Denote by $S$ the surface so obtained. 
The boundary of $S$ gives a two component link $l_0 \cup l'_0$ 
in $S^3$. 
Add five trivial components $l_1, l_2, l_3, l_4, l_5$ 
piercing each plumbed annulus. 
Add more five trivial components $l'_1, l'_2, l'_3, l'_4, l'_5$ 
such that they correspond to the cores of plumbed annuli 
and $l'_k$ is isotopic to the meridian of $l_k$ ($1 \leq k \leq 5$) 
in the exterior of 
the link $l_0 \cup l'_0 \cup l_1 \cup l_2 \cup l_3 \cup l_4 \cup l_5$. 
See Figure \ref{fig3}. 

\unitlength=0.9mm
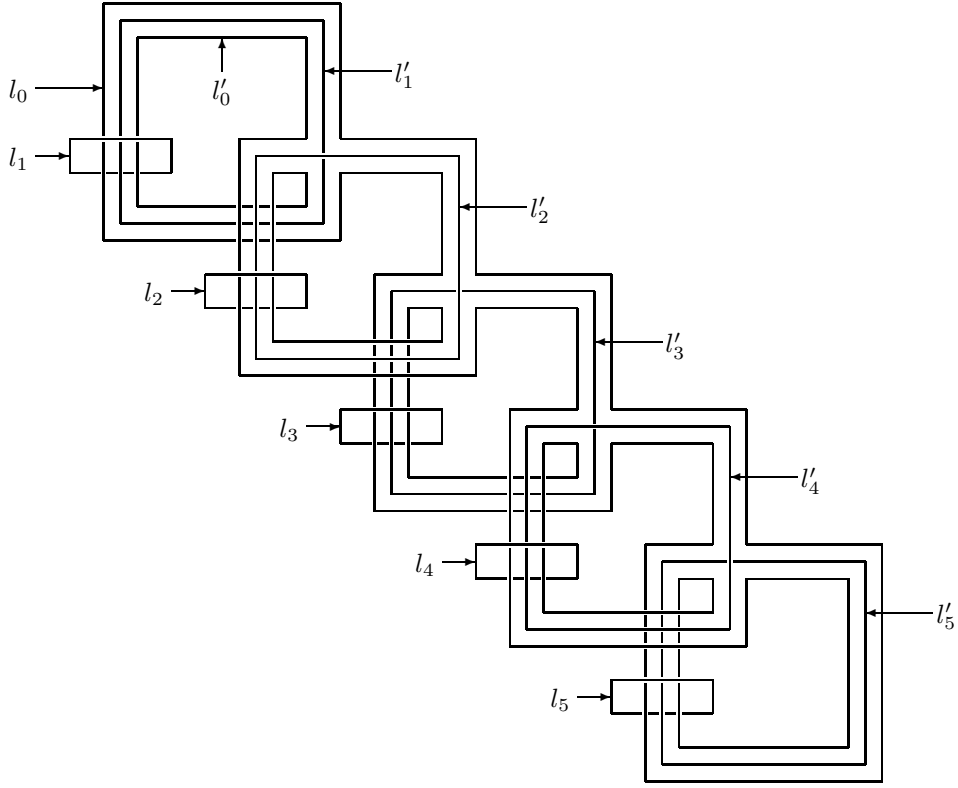
\begin{figure}[ht]
 \begin{picture}(120,120)
	\linethickness{0.2mm}
		\newsavebox{\pierce}% Piercing
		\savebox{\pierce}(15,5) 
		{ 
		\begin{picture}(15,5)
			\put (0,5){\line(1,0){15}}
			\put (0,5){\line(0,-1){5}}
			\put (15,5){\line(0,-1){5}}
			\put (0,0){\line(1,0){4.5}}
			\put (5.5,0){\line(1,0){1.5}}
			\put (8,0){\line(1,0){1.5}}
			\put (10.5,0){\line(1,0){4.5}}
		\end{picture} 
		} 
		\newsavebox{\HopfA}	% HopfA
		\savebox{\HopfA}(35,35) 
		{ 
		\begin{picture}(35,35)
			\put (0,35){\line(1,0){35}}		% Outer
			\put (35,35){\line(0,-1){20}}
			\put (35,10){\line(0,-1){10}}
			\put (35,0){\line(-1,0){9.5}}
			\put (19.5,0){\line(-1,0){19.5}}
			\put (0,0){\line(0,1){14.5}}
			\put (0,15.5){\line(0,1){19.5}}
			\put (5,30){\line(1,0){25}}		% Inner
			\put (30,30){\line(0,-1){15}}
			\put (30,10){\line(0,-1){5}}
			\put (30,5){\line(-1,0){4.5}}
			\put (19.5,5){\line(-1,0){14.5}}
			\put (5,5){\line(0,1){9.5}}
			\put (5,15.5){\line(0,1){14.5}}
			\put (2.5,32.5){\line(1,0){30}}	% Middle
			\put (32.5,32.5){\line(0,-1){19.5}}
			\put (32.5,12){\line(0,-1){9.5}}
			\put (32.5,2.5){\line(-1,0){7}}
			\put (19.5,2.5){\line(-1,0){17}}
			\put (2.5,2.5){\line(0,1){12}}
			\put (2.5,15.5){\line(0,1){17}}
			\multiput(24.5,0)(0,2.5){3}{\line(-1,0){1.5}}% parallel part
			\multiput(22,0)(0,2.5){3}{\line(-1,0){1.5}}
		\end{picture} 
		} 
		\newsavebox{\HopfC}	% HopfC
		\savebox{\HopfC}(35,35) 
		{ 
		\begin{picture}(35,35)
			\put (0,35){\line(1,0){10}}		% Outer
			\put (15,35){\line(1,0){20}}
			\put (35,35){\line(0,-1){20}}
			\put (35,10){\line(0,-1){10}}
			\put (35,0){\line(-1,0){9.5}}
			\put (19.5,0){\line(-1,0){19.5}}
			\put (0,0){\line(0,1){14.5}}
			\put (0,15.5){\line(0,1){4}}
			\put (0,25.5){\line(0,1){9.5}}
			\put (5,30){\line(1,0){5}}		% Inner
			\put (15,30){\line(1,0){15}}
			\put (30,30){\line(0,-1){15}}
			\put (30,10){\line(0,-1){5}}
			\put (30,5){\line(-1,0){4.5}}
			\put (19.5,5){\line(-1,0){14.5}}
			\put (5,5){\line(0,1){9.5}}
			\put (5,15.5){\line(0,1){4}}
			\put (5,25.5){\line(0,1){4.5}}
			\put (2.5,32.5){\line(1,0){9.5}}	% Middle
			\put (13,32.5){\line(1,0){19.5}}
			\put (32.5,32.5){\line(0,-1){19.5}}
			\put (32.5,12){\line(0,-1){9.5}}
			\put (32.5,2.5){\line(-1,0){7}}
			\put (19.5,2.5){\line(-1,0){17}}
			\put (2.5,2.5){\line(0,1){12}}
			\put (2.5,15.5){\line(0,1){4}}
			\put (2.5,25.5){\line(0,1){7}}
			\multiput(24.5,0)(0,2.5){3}{\line(-1,0){1.5}}% parallel part
			\multiput(22,0)(0,2.5){3}{\line(-1,0){1.5}}
			\multiput(0,20.5)(2.5,0){3}{\line(0,1){1.5}}
			\multiput(0,23)(2.5,0){3}{\line(0,1){1.5}}
		\end{picture} 
		} 
		\newsavebox{\HopfE}	% HopfE
		\savebox{\HopfE}(35,35) 
		{ 
		\begin{picture}(35,35)
			\put (0,35){\line(1,0){10}}		% Outer
			\put (15,35){\line(1,0){20}}
			\put (35,35){\line(0,-1){35}}
			\put (35,0){\line(-1,0){35}}
			\put (0,0){\line(0,1){14.5}}
			\put (0,15.5){\line(0,1){4}}
			\put (0,25.5){\line(0,1){9.5}}
			\put (5,30){\line(1,0){5}}		% Inner
			\put (15,30){\line(1,0){15}}
			\put (30,30){\line(0,-1){25}}
			\put (30,5){\line(-1,0){25}}
			\put (5,5){\line(0,1){9.5}}
			\put (5,15.5){\line(0,1){4}}
			\put (5,25.5){\line(0,1){4.5}}
			\put (2.5,32.5){\line(1,0){9.5}}	% Middle
			\put (13,32.5){\line(1,0){19.5}}
			\put (32.5,32.5){\line(0,-1){30}}
			\put (32.5,2.5){\line(-1,0){30}}
			\put (2.5,2.5){\line(0,1){12}}
			\put (2.5,15.5){\line(0,1){4}}
			\put (2.5,25.5){\line(0,1){7}}
			\multiput(0,20.5)(2.5,0){3}{\line(0,1){1.5}}% parallel part
			\multiput(0,23)(2.5,0){3}{\line(0,1){1.5}}
		\end{picture} 
		} 
		\newsavebox{\overHopf}	% OverHopf
		\savebox{\overHopf}(35,35) 
		{ 
		\begin{picture}(35,35)
			\put (0,35){\line(1,0){10}}		% Outer
			\put (15,35){\line(1,0){20}}
			\put (35,35){\line(0,-1){20}}
			\put (35,10){\line(0,-1){10}}
			\put (35,0){\line(-1,0){35}}
			\put (0,0){\line(0,1){14.5}}
			\put (0,15.5){\line(0,1){19.5}}
			\put (5,30){\line(1,0){5}}		% Inner
			\put (15,30){\line(1,0){15}}
			\put (30,30){\line(0,-1){15}}
			\put (30,10){\line(0,-1){5}}
			\put (30,5){\line(-1,0){25}}
			\put (5,5){\line(0,1){9.5}}
			\put (5,15.5){\line(0,1){14.5}}
			\put (2.5,32.5){\line(1,0){30}}	% Middle
			\put (32.5,32.5){\line(0,-1){30}}
			\put (32.5,2.5){\line(-1,0){30}}
			\put (2.5,2.5){\line(0,1){12}}
			\put (2.5,15.5){\line(0,1){17}}
		\end{picture} 
		} 
	\multiput(0,95)(20,-20){5}{\usebox{\pierce}}  % Pierces
	\put (5,85){\usebox{\HopfA}}
	\put (45,45){\usebox{\HopfC}}
	\put (85,5){\usebox{\HopfE}}
	\multiput(25,65)(40,-40){2}{\usebox{\overHopf}}  % Hopf B & D
	\linethickness{0.1mm}
	\put (-5, 107.5){\vector(1,0){10}}
	\put (-9,106){$l_0$}
	\put (22.5,110){\vector(0,1){5}}
	\put (21,106){$l'_0$}
	\multiput(-5, 97.5)(20,-20){5}{\vector(1,0){5}}
	\put (-9,96){$l_1$}
	\put (11,76){$l_2$}
	\put (31,56){$l_3$}
	\put (51,36){$l_4$}
	\put (71,16){$l_5$}
	\multiput(47.5, 110)(20,-20){5}{\vector(-1,0){10}}
	\put (48,108.5){$l'_1$}
	\put (68,88.5){$l'_2$}
	\put (88,68.5){$l'_3$}
	\put (108,48.5){$l'_4$}
	\put (128,28.5){$l'_5$}
 \end{picture}
 \caption{The link $L$}\label{fig3}
\end{figure}

Let 
$L(r_0, r'_0, r_1, r_2, r_3, r_4, r_5, r'_1, r'_2, r'_3, r'_4, r'_5)$ 
denote the $3$-manifold obtained by Dehn surgery 
along the link $L$ with 
surgery coefficient $r_i, r'_j$ for $l_i, l'_j$ ($0 \leq i,j \leq 5$). 

\begin{claim}
The manifold $M_{\boldsymbol{\varepsilon},n}$ is homeomorphic to 
$L(0, 0, 0, 0, 0, 0, 0, 
\varepsilon_1, \varepsilon_2, \varepsilon_3, \varepsilon_4, 1/n)$. 
\end{claim}

\begin{proof}
First we consider the sublink 
$L_0 = l_0 \cup l'_0 \cup l_1 \cup l_2 \cup l_3 \cup l_4 \cup l_5$. 
Then it is easily seen that $L_0 (0, 0, 0, 0, 0, 0, 0)$ is 
homeomorphic to the product bundle $F_2 \times S^1$. 
See \cite{I} for example. 

The remaining components $l'_1, l'_2, l'_3, l'_4, l'_5$ 
are regarded as in the surgered manifold. 
These can be isotoped to lie on surface fibers, and 
we can assume that 
$l'_k$ is projected to $C_k$ on $F_2$ 
by the natural projection $F_2 \times S^1 \to F_2$. 

Now we regard $F_2 \times S^1$ as 
a surface bundle $M_{id.}$ with trivial monodromy. 
Remark that the preferred longitudes of 
the components $l'_1, l'_2, l'_3, l'_4, l'_5$ in $S^3$ 
are coincident with the longitudes induced by 
the bundle structure of $F_2 \times S^1$. 
Then it is known that 
the $1/n$-surgery on such a component $l'_k$ 
yields the manifold $M_{{D_k}^n}$ (see \cite{St, I}). 

Note that the components $l'_1, l'_3, l'_5$ are regarded as 
lying on the same fiber $F_2 \times \{ 1/3 \}$ and 
also $l'_2, l'_4$ as lying on $F_2 \times \{ 2/3 \}$. 
Thus 
$L(0, 0, 0, 0, 0, 0, 0, 
\varepsilon_1, \varepsilon_2, \varepsilon_3, \varepsilon_4, 1/n)$ 
is homeomorphic to 
$M_{{D_1}^{\varepsilon_1} \circ {D_3}^{\varepsilon_3} \circ 
{D_5}^n \circ {D_2}^{\varepsilon_2} \circ {D_4}^{\varepsilon_4} }$. 
We note that
\begin{eqnarray*}
{D_1}^{\varepsilon_1} \circ {D_3}^{\varepsilon_3} \circ 
{D_5}^n \circ {D_2}^{\varepsilon_2} \circ {D_4}^{\varepsilon_4} 
& \sim &
{D_1}^{\varepsilon_1} \circ {D_3}^{\varepsilon_3} \circ 
{D_5}^n \circ {D_4}^{\varepsilon_4} \circ {D_2}^{\varepsilon_2} \\
&\sim&
{D_2}^{\varepsilon_2} \circ {D_1}^{\varepsilon_1} \circ {D_3}^{\varepsilon_3} \circ 
{D_5}^n \circ {D_4}^{\varepsilon_4} \\
&\sim&
{D_5}^n \circ {D_2}^{\varepsilon_2} \circ {D_1}^{\varepsilon_1} \circ {D_3}^{\varepsilon_3} \circ {D_4}^{\varepsilon_4} \\
&\sim&
{D_2}^{\varepsilon_2} \circ 
{D_1}^{\varepsilon_1} \circ {D_3}^{\varepsilon_3} \circ 
{D_4}^{\varepsilon_4} \circ {D_5}^n
\sim f_{\boldsymbol{\varepsilon},n}, 
\end{eqnarray*}
where $\sim$ denotes suitable isotopy or conjugates. 
Thus 
$M_{{D_1}^{\varepsilon_1} \circ {D_3}^{\varepsilon_3} \circ 
{D_5}^n \circ {D_2}^{\varepsilon_2} \circ {D_4}^{\varepsilon_4} }$ 
is homeomorphic to our manifold $M_{\boldsymbol{\varepsilon},n}$. 

\end{proof}

By using the well-known modifications (see \cite{R} for example), 
we can simplify the surgery description obtained above to 
the one depicted in Figure \ref{fig4}. 

\unitlength=1.3mm
\begin{figure}[ht]
 \begin{picture}(50,45)
		\newsavebox{\uplink}	% uplink
		\savebox{\uplink}(5,25) 
		{ 
		\begin{picture}(5,25)
			\qbezier (0,20.5)(2.5,25)(5,20.5)
			\qbezier (0,4.5)(2.5,0)(5,4.5)
			\put (5,4.5){\line(0,1){16}}
			\put (0,5.5){\line(0,1){9}}
			\put (0,15.5){\line(0,1){4}}
		\end{picture} 
		} 
		\newsavebox{\lowlink}		% lowlink
		\savebox{\lowlink}(5,20) 
		{ 
		\begin{picture}(5,20)
			\qbezier (0,15.5)(2.5,20)(5,15.5)
			\qbezier (0,4.5)(2.5,0)(5,4.5)
			\put (5,4.5){\line(0,1){11}}
			\put (0,5.5){\line(0,1){9}}
		\end{picture} 
		} 
	\put (5,40){\oval(10,10)[l]}	% top component
	\put (5,45){\line(1,0){40}}
	\put (45,40){\oval(10,10)[r]}
	\put (5,35){\line(1,0){9.5}}
	\put (15.5,35){\line(1,0){19}}
	\put (35.5,35){\line(1,0){9.5}}
	\put (51,40){$1/n$}
	\multiput(10,15)(20,0){2}{\usebox{\uplink}}% augmented compo.
	\put (15,40){$-\varepsilon_1$}
	\put (35,40){$-\varepsilon_3$}
	\multiput(20,5)(20,0){2}{\usebox{\lowlink}}
	\put (15,4){$-\varepsilon_2$}
	\put (35,4){$-\varepsilon_4$}
	\put (5,15){\oval(10,30)[l]}		% lower component
	\put (5,30){\line(1,0){9.5}}
	\put (15.5,30){\line(1,0){19}}
	\put (35.5,30){\line(1,0){10}}
	\multiput(46,25)(0,-20){2}{\oval(10,10)[r]}
	\multiput(44.5,20)(-10,0){3}{\line(-1,0){9}}
	\put (14.5,20){\line(-1,0){4.5}}
	\put (10,15){\oval(10,10)[l]}
	\put (10,10){\line(1,0){14.5}}
	\put (25.5,10){\line(1,0){19}}
	\put (46,0){\line(-1,0){41}}
	\put (52,25){$0$}
 \end{picture}
 \caption{}\label{fig4}
\end{figure}

To find a Heegaard surface, 
we perform surgeries on the four components 
with surgery coefficients $ - \varepsilon_i $ ($1 \leq i \leq 4$) 
to obtain a two component link. 
For example, the link corresponding to 
the case 
$(\varepsilon_1, \varepsilon_2, \varepsilon_3, \varepsilon_4)
=(-1, 1,-1,-1)$ is illustrated in Figure \ref{fig5}. 

\unitlength=1.3mm
\begin{figure}[ht]
 \begin{picture}(90,25)
		\newsavebox{\crp}	% Crossing "+" as braid 
		\savebox{\crp}(5,5) 
		{ 
		\begin{picture}(5,5)
			\put (0,0){\line(1,1){5}} 
			\multiput(0,5)(3, -3){2}{\line(1,-1){2}} 
		\end{picture}
		} 
		\newsavebox{\crm}	% Crossing "-" as braid 
		\savebox{\crm}(5,5) 
		{ 
		\begin{picture}(5,5)
			\put (0,5){\line(1,-1){5}} 
			\multiput(0,0)(3, 3){2}{\line(1,1){2}} 
		\end{picture}
		} 
		\newsavebox{\delt}	% Delta
		\savebox{\delt}(15,10) 
		{ 
		\begin{picture}(15,10)
			\multiput(0,0)(10, 0){2}{\usebox{\crp}} 
			\multiput(0,10)(10, 0){2}{\line(1,0){5}}
			\put (5,0){\line(1,0){5}} 
			\put (5,5){\usebox{\crp}} 
		\end{picture}
		} 
		\newsavebox{\doubledelta}	% doubledelta
		\savebox{\doubledelta}(30,10) 
		{ 
		\begin{picture}(30,10)
			\multiput(0,0)(15, 0){2}{\usebox{\delt}} 
		\end{picture}
		} 
	\multiput(5,0)(0,25){2}{\line(1,0){80}}
	\multiput(5,5)(40,0){2}{\line(1,0){30}}	\multiput(35,15)(40,0){2}{\line(1,0){10}}
	\multiput(35,20)(40,0){2}{\line(1,0){10}}
	\multiput(5,10)(40,0){2}{\usebox{\doubledelta}}
	\multiput(75,5)(5,0){2}{\usebox{\crp}}
	\multiput(35,5)(5,0){2}{\usebox{\crm}}
	\put (5,7.5){\oval(4,5)[l]}
	\put (5,7.5){\oval(10,15)[l]}
	\put (5,22.5){\oval(10,5)[l]}
	\multiput(85,2.5)(0,10){3}{\oval(10,5)[r]}
	\linethickness{0.3mm}
	\put (45,5){\line(0,1){5}}
 \end{picture}
 \caption{}\label{fig5}
\end{figure}

Consider the boundary of the regular neighborhood of 
the non-trivial component of the link 
together with the thickened arc illustrated in Figure \ref{fig5}. 
Then, by manipulating the figure, 
one can check that the surface 
separates the link complement into 
two homeomorphic compression bodies. 
This implies the surface becomes 
a genus two Heegaard surface of the surgered manifold, 
i.e., of $M_{\boldsymbol{\varepsilon},n}$. 
Thus the Heegaard genus of $M_{\boldsymbol{\varepsilon},n}$ is at most two. 

On the other hand, since our manifold $M_{\boldsymbol{\varepsilon},n}$ 
fibers over the circle with genus two fiber, 
its fundamental group contains a non-abelian surface subgroup. 
This implies that $M_{\boldsymbol{\varepsilon},n}$ fails to 
be $S^3$, $S^2 \times S^1$, lens spaces; 
in particular, it fails to have Heegaard genus one. 
Therefore the Heegaard genus of $M_{\boldsymbol{\varepsilon},n}$ 
is shown to be two. 
\end{proof}

This completes the proof of Theorem \ref{thm1}. 

As we stated in Section 1, 
most of our manifolds $\{ M_{\boldsymbol{\varepsilon},n} \}$ are 
shown to be hyperbolic. 

\begin{proposition}
With the same notations as in Theorem $\ref{thm1}$, 
\begin{enumerate}
\item
for $n=0$, 
the manifolds $\{ M_{\boldsymbol{\varepsilon},0} \}$ are hyperbolic unless $\varepsilon_1 =\varepsilon_2 =\varepsilon_3 =\varepsilon_4$, 
\item
for each $\boldsymbol{\varepsilon}$, 
the manifolds $\{ M_{\boldsymbol{\varepsilon},n} \}$ are hyperbolic 
with at most five exceptions. 
\end{enumerate}
\end{proposition}

\begin{proof}
We can see directly from Figure \ref{fig4} that 
$M_{\boldsymbol{\varepsilon},0}$ is obtained by $0$-surgery 
along the two-bridge knot with Conway's normal form 
$[2\varepsilon_4, -2\varepsilon_3, 2\varepsilon_2, -2\varepsilon_1]$. 
Then, by the classification of 
exceptional surgeries on two-bridge knots \cite{BrWu}, 
we see that $M_{\boldsymbol{\varepsilon},0}$ is hyperbolic other than 
the corresponding knot is of genus one or a torus knot. 
All our knots are of genus two, and so, 
the knot we have to exclude is the (2,5)-torus knot. 
This corresponds to the case where 
$\varepsilon_1 =\varepsilon_2 =\varepsilon_3 =\varepsilon_4$. 

Next, note that the homeomorphism 
$f_{\boldsymbol{\varepsilon},0}={D_2}^{\varepsilon_2} \circ 
{D_1}^{\varepsilon_1} \circ {D_3}^{\varepsilon_3} \circ {D_4}^{\varepsilon_4}$ 
of $F_2$ is irreducible. 
Because, as we claimed above, 
each manifold $M_{\boldsymbol{\varepsilon},0}$ is 
obtained by 0-surgery along a two-bridge knot of genus two, 
but such a $3$-manifold is always atoroidal by \cite{BrWu}. 
Thus $({D_2}^{\varepsilon_2} \circ 
{D_1}^{\varepsilon_1} \circ {D_3}^{\varepsilon_3} \circ {D_4}^{\varepsilon_4}, C_5)$ 
fills $F_2$ up, and so, we can apply \cite[Theorem 2.6]{BGZ}. 
This concludes the second assertion of the proposition. 
\end{proof}

%=======================================================
\section{Fibered knots}

As a byproduct of Lemma \ref{lem1} and Lemma \ref{lem2}, 
we obtain an infinite family of hyperbolic knots 
in the $3$-sphere $S^3$ whose knot groups are of rank two. 

\begin{proposition}\label{prop-knots}
Let $n$ be an arbitrary integer and 
$\varepsilon_i$ either $+1$ or $-1$ for $i = 1,2,4$ and 
$\varepsilon_3 = -\varepsilon_1$. 
Then the Montesinos knot 
$M \left( \frac{1}{2 \varepsilon_1}, 
\frac{2 \varepsilon_2}{4 \varepsilon_2 n -1}, 
\frac{2 \varepsilon_4}{4 \varepsilon_3 \varepsilon_4 -1} \right)$ 
is always hyperbolic, genus two, fibered knots in $S^3$ 
with rank two knot group. 
\end{proposition}

\begin{proof}
We consider the once-punctured surface bundles $M'_{\boldsymbol{\varepsilon},n}$ 
defined in the same way as the manifold $M_{\boldsymbol{\varepsilon},n}$ 
in Theorem \ref{thm1}. 
Then its fundamental group is shown to be of rank two 
in the same way as proving Lemma \ref{lem1}. 
As we explained in the proof of Lemma \ref{lem2}, 
the manifold $M'_{\boldsymbol{\varepsilon},n}$ has obtained by surgery 
on the link depicted in Figure \ref{fig3}. 
The surgery coefficients are the same as 
for $M_{\boldsymbol{\varepsilon},n}$ except for that on $l_0$. 
On $l_0$, no filling is performed, i.e., 
it corresponds to a toral boundary. 
This surgery description can be simplified, 
under the assumption that $\varepsilon_3 = -\varepsilon_1$, 
to the one illustrated in Figure \ref{fig6}. 

\unitlength=1mm
\begin{figure}[ht]
 \begin{picture}(55,40)
		\newsavebox{\crpn}	% Crossing "+" as braid 
		\savebox{\crpn}(5,5) 
		{ 
		\begin{picture}(5,5)
			\put (0,0){\line(1,1){5}} 
			\multiput(0,5)(3, -3){2}{\line(1,-1){2}} 
		\end{picture}
		} 
		\newsavebox{\crmn}	% Crossing "-" as braid 
		\savebox{\crmn}(5,5) 
		{ 
		\begin{picture}(5,5)
			\put (0,5){\line(1,-1){5}} 
			\multiput(0,0)(3, 3){2}{\line(1,1){2}} 
		\end{picture}
		} 
		\newsavebox{\aug}	% augment
		\savebox{\aug}(15,10) 
		{ 
		\begin{picture}(15,10)
			\multiput(0,0)(0, 5){2}{\usebox{\crmn}} 
			\multiput(10,0)(0, 5){2}{\usebox{\crpn}} 
			\multiput(5,0)(0, 10){2}{\line(1,0){5}}
		\end{picture}
		} 
		\newsavebox{\doubleaug}	% double augment
		\savebox{\doubleaug}(15,30) 
		{ 
		\begin{picture}(15,30)
			\multiput(0,0)(10,0){2}{\usebox{\crpn}}
			\multiput(5,0)(0,5){2}{\line(1,0){5}}
			\multiput(0,5)(15,0){2}{\line(0,1){5}}
			\multiput(0,10)(10,0){2}{\usebox{\crmn}}
			\multiput(5,10)(0,5){2}{\line(1,0){5}}
			\multiput(0,15)(15,0){2}{\line(0,1){5}}
			\put (0,20){\usebox{\aug}} 
		\end{picture}
		} 
	\multiput(0,35)(55,0){2}{\line(0,1){5}}
	\multiput(0,0)(55,0){2}{\line(0,1){5}}
	\multiput(0,0)(0,40){2}{\line(1,0){55}}
	\multiput(0,5)(0,20){2}{\line(0,1){10}}
	\multiput(15,5)(0,20){2}{\line(0,1){10}}
	\multiput(15,5)(0,30){2}{\line(1,0){5}}
	\put (0,15){\usebox{\aug}}
	\multiput(35,5)(0,30){2}{\line(1,0){5}}
	\multiput(20,5)(20,0){2}{\usebox{\doubleaug}}
	\put (-6,20){$- \varepsilon_1$}
	\put (25,2.5){$- \varepsilon_2$}
	\put (55.5,30){$- \varepsilon_3 = \varepsilon_1$}
	\put (55.5,12){$- \varepsilon_4$}
	\put (25,36){$1/n$}
 \end{picture}
 \caption{}\label{fig6}
\end{figure}

By performing surgeries, we obtain 
the Montesinos knot 
$M \left( \frac{1}{2 \varepsilon_1}, 
\frac{2 \varepsilon_2}{4 \varepsilon_2 n -1}, 
\frac{2 \varepsilon_4}{4 \varepsilon_3 \varepsilon_4 -1} \right)$ 
from the unlabeled component. 

\unitlength=1.2mm
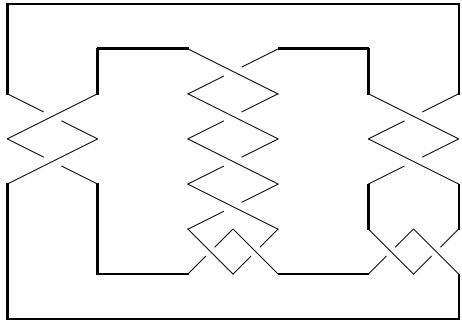
\begin{figure}[ht]
 \begin{picture}(50,35)
		\newsavebox{\ptwist}	% +twist
		\savebox{\ptwist}(10,10) 
		{ 
		\begin{picture}(10,10)
			\multiput(0,0)(0,5){2}{\line(2,1){10}} 
			\multiput(10,0)(-6, 3){2}{\line(-2,1){4}} 
			\multiput(10,5)(-6, 3){2}{\line(-2,1){4}} 
		\end{picture}
		} 
		\newsavebox{\mtwist}	% -twist
		\savebox{\mtwist}(10,10) 
		{ 
		\begin{picture}(10,10)
			\multiput(10,0)(0,5){2}{\line(-2,1){10}} 
			\multiput(0,0)(6, 3){2}{\line(2,1){4}} 
			\multiput(0,5)(6, 3){2}{\line(2,1){4}} 
		\end{picture}
		} 
		\newsavebox{\hook}	% hook
		\savebox{\hook}(10,5) 
		{ 
		\begin{picture}(10,5)
			\multiput(10,0)(-5,0){2}{\line(-1,1){5}} 
			\multiput(0,0)(3, 3){2}{\line(1,1){2}} 
			\multiput(5,0)(3, 3){2}{\line(1,1){2}} 
		\end{picture}
		} 
	\multiput(0,30)(50,0){2}{\line(0,1){5}}
	\multiput(0,0)(50,0){2}{\line(0,1){5}}
	\multiput(0,0)(0,35){2}{\line(1,0){50}}
	\multiput(10,5)(0,25){2}{\line(1,0){10}}
	\multiput(30,5)(0,25){2}{\line(1,0){10}}
	\put (0,15){\usebox{\ptwist}}
	\multiput(0,5)(10,0){2}{\line(0,1){10}}
	\multiput(0,25)(10,0){2}{\line(0,1){5}}
	\put (40,15){\usebox{\mtwist}}
	\multiput(20,10)(0,10){2}{\usebox{\mtwist}}
	\multiput(20,5)(20,0){2}{\usebox{\hook}}
	\multiput(40,10)(10,0){2}{\line(0,1){5}}
	\multiput(40,25)(10,0){2}{\line(0,1){5}}
 \end{picture}
 \caption{
	$( \varepsilon_1, \varepsilon_2, \varepsilon_3, \varepsilon_4, n ) 
	 = (1, 1, -1, 1, 2)$ }
 \label{fig7}
\end{figure}

Consequently the knot groups of our knots, 
which are isomorphic to 
the fundamental groups of $M'_{\boldsymbol{\varepsilon},n}$, 
are shown to be of rank two. 

The fact that the knots are of genus two and fibered 
can be checked directly. 
Consider the diagram of the knot 
which is naturally obtained from Figure \ref{fig6} 
by performing surgeries 
(an example is depicted in Figure \ref{fig7}). 
Let $S$ be the Seifert surface obtained 
from this diagram by Seifert's algorithm. 
One can see that 
two Hopf bands are deplumbed from $S$, and 
the surface so obtained is modified 
to the connected sum of two Hopf bands by Stallings twists. 
This shows that $S$ is a fiber surface. 
See \cite{Ha} for example. 
Since $S$ is of genus two and 
any fiber surface is minimal genus 
(see \cite{BuZ} for example), 
all our knots are of genus two. 

While it is known which Montesinos knots are hyperbolic \cite{BS, Oe}, 
here we show that our knots are hyperbolic directly. 
First note that our knots are all unknotting number one 
as the diagram considered above shows. 
By \cite{KM}, 
no torus knot of genus two has unknotting number one, 
and so, none of our knots are torus knots. 
Also note that our knots are all at most three bridge 
as the diagram shows. 
This implies that if our knots are satellite, 
then they must be composite (\cite{Schu}), 
but it is impossible by \cite{Scha2}. 
\end{proof}

In fact, they are all \textit{tunnel number one knots}, 
meaning that, the exteriors admit a genus two Heegaard surface. 
Their unknotting tunnels place 
in the corresponding position to 
the thickened arc illustrated in Figure \ref{fig5}. 

%=======================================================
\section{Computer Experiments}

In this section, we show some examples of $3$-manifolds with $2$-generator
fundamental groups which are found by computer experiments.
All examples are surface bundles over the circle with genus two surface fiber.

Recall that the Dehn twists $D_1, D_2, \cdots, D_5$ generate the mapping class group
of a closed surface of genus two.
We fix this generating system and describe the monodromy of surface bundles
in terms of the word of these generators.

In the first experiment, we compute the rank of the fundamental
group for all closed/once-punctured surface bundles with monodromy up to word length 5.
There is a table of the representative elements of 
all conjugacy classes of monodromies up to word length 4 in \cite{AT}.
%there is a table of representative elements of 
%conjugacy classes of the monodromy up to length 5.
As an extension of the result, 
we can obtain such a list up to length 5, 
which is now available at 
\texttt{http://www.is.titech.ac.jp/\~{}takasawa/MCG/}.
In fact, there are 172 nontrivial representative elements. 
To compute the rank of the fundamental group of the manifold, 
we implement the algorithm in \cite{I}, which constructs a surgery description
of the surface bundle from the monodromy.
Then we use SnapPea \cite{SnapPea} to compute the fundamental group of
the manifold.
Note that in general, it is difficult to determine the rank of the fundamental group.
However, in this case, 
if there is a presentation which has just 2-generator, 
we can say the rank is 2.
The result of the experiment is shown in Table \ref{Table1}.
%There are 34 distinct examples which have $2$-generator fundamental groups.
There are 30 distinct examples which have $2$-generator fundamental groups.
For some examples 
the first Betti number $\beta_1$ equals to $2$,
however they are
torus sums of two Seifert manifolds.

\begin{table}[hbt]
\caption{}\label{Table1}
\begin{tabular}{|l|l|l|l|l|l|l|}
\hline
		&	\multicolumn{3}{c|}{closed} & \multicolumn{3}{c|}{once-punctured} \\
monodromy	&	volume &	rank of $\pi_1$	& $\beta_1$  &
volume &	rank of $\pi_1$	&	$\beta_1$  \\
\hline
$  D_1  D_2  D_3   $ & 0.000002 & 2 & 2 & 0.000077 & 2 & 2\\
$  D_1  D_2  D_3^{-1}   $ & 3.663862 & 2 & 2 & 3.663872 & 2 & 2\\
$  D_1  D_2^{-1}  D_3   $ & 5.333490 & 2 & 2 & 5.333605 & 2 & 2\\
$  D_1  D_2  D_3  D_4   $ & 0.000000 & 2 & 1 & 0.000000 & 2 & 1\\
$  D_1  D_2  D_3  D_4^{-1}   $ & 4.400833 & 2 & 1 & 3.770830 & 2 & 1\\
%$  D_1  D_2  D_3^{-1}  D_4   $ & 7.084926 & 2 & 1 & 6.180274 & 2 & 1\\
$  D_1  D_2  D_3^{-1}  D_4^{-1}   $ & 5.693021 & 2 & 1 & 4.059766 & 2 & 1\\
$  D_1  D_2^{-1}  D_3  D_4   $ & 7.084926 & 2 & 1 & 6.180274 & 2 & 1\\
$  D_1  D_2^{-1}  D_3  D_4^{-1}   $ & 8.935857 & 2 & 1 & 7.646593 & 2 & 1\\
$  D_1  D_2^{-1}  D_3^{-1}  D_4   $ & 7.643375 & 2 & 1 & 6.332667 & 2 & 1\\
$  D_1^2  D_2  D_3  D_4   $ & 0.000000 & 2 & 1 & 0.000000 & 2 & 1\\
$  D_1^2  D_2  D_3  D_4^{-1}   $ & 4.921483 & 2 & 1 & 3.970290 & 2 & 1\\
$  D_1^2  D_2  D_3^{-1}  D_4   $ & 7.967261 & 3 & 1 & 6.788090 & 2 & 1\\
$  D_1^2  D_2  D_3^{-1}  D_4^{-1}   $ & 6.577943 & 2 & 1 & 4.765940 & 2 & 1\\
$  D_1^2  D_2^{-1}  D_3  D_4   $ & 8.118328 & 2 & 1 & 7.023949 & 2 & 1\\
$  D_1^2  D_2^{-1}  D_3  D_4^{-1}   $ & 9.951958 & 2 & 1 & 8.550620 & 2 & 1\\
$  D_1^2  D_2^{-1}  D_3^{-1}  D_4   $ & 8.554563 & 2 & 1 & 7.180680 & 2 & 1\\
$  D_1^2  D_2^{-1}  D_3^{-1}  D_4^{-1}   $ & 5.231154 & 2 & 1 & 4.725229 & 2 & 1\\
$  D_1  D_2  D_1^{-1}  D_2  D_3^{-1}   $ & 5.333490 & 2 & 2 & 5.333490 & 2 & 2\\
$  D_1  D_2^2  D_3  D_4   $ & 4.464659 & 2 & 1 & 3.853456 & 2 & 1\\
$  D_1  D_2^2  D_3  D_4^{-1}   $ & 6.746042 & 3 & 1 & 5.137941 & 2 & 1\\
%$  D_1  D_2  D_3^2  D_4   $ & 4.464659 & 2 & 1 & 3.853456 & 2 & 1\\
$  D_1  D_2  D_3  D_4  D_5   $ & 0.000000 & 2 & 1 & 0.000000 & 2 & 1\\
$  D_1  D_2  D_3  D_4  D_5^{-1}   $ & 4.056860 & 2 & 1 & 3.177293 & 2 & 1\\
$  D_1  D_2  D_3  D_4^{-1}  D_5   $ & 6.771750 & 2 & 1 & 5.563668 & 2 & 1\\
$  D_1  D_2  D_3  D_4^{-1}  D_5^{-1}   $ & 4.124903 & 2 & 1 & 0.790429 & 2 & 1\\
$  D_1  D_2  D_3^{-1}  D_4  D_5   $ & 7.746275 & 2 & 1 & 6.551743 & 2 & 1\\
$  D_1  D_2  D_3^{-1}  D_4  D_5^{-1}   $ & 8.602031 & 2 & 1 & 6.965760 & 2 & 1\\
$  D_1  D_2  D_3^{-1}  D_4^{-1}  D_5   $ & 7.406768 & 3 & 1 & 5.333490 & 2 & 1\\
%$  D_1  D_2^{-1}  D_3  D_4  D_5   $ & 6.771750 & 2 & 1 & 5.563668 & 3 & 1\\
$  D_1  D_2^{-1}  D_3  D_4  D_5^{-1}   $ & 9.250556 & 2 & 1 & 7.517690 & 3 & 1\\
$  D_1  D_2^{-1}  D_3  D_4^{-1}  D_5   $ & 10.649781 & 3 & 1 & 8.793346 & 2 & 1\\
%$  D_1  D_2^{-1}  D_3^{-1}  D_4  D_5   $ & 7.406768 & 2 & 1 & 5.333490 & 3 & 1\\
$  D_1  D_2^{-1}  D_3^{-1}  D_4^{-1}  D_5   $ & 6.783714 & 2 & 1 & 5.333490 & 3 & 1\\
\hline
\end{tabular}
\end{table}

In the second experiment, we searched over 100,000 randomly generated 
words up to word length 20 and
we found huge number of examples of hyperbolic surface bundles 
with rank two fundamental groups. 
Unfortunately, for many of such manifolds, 
we have no practical algorithm to detect their Heegaard genus, 
and at present, 
we do not have examples of Heegaard genus more than two. 

Here we exhibit single example we found, 
which seems independently interesting. 
The closed surface bundle $M_f$ 
with monodromy $ f = D_1^2 D_3^3 D_5^{-1} D_4^{-1} D_3^{-1} D_2 $ 
has $2$-generator fundamental group and the first Betti number 
of $M_f$ is also $2$ and $M_f$ admits a hyperbolic structure. 
The fact that the first Betti number is two implies that 
it admits more than one surface bundle structures. 
Unfortunately, we could not determined the Heegaard genus of $M_f$.

%=======================================================
\bibliographystyle{amsplain}

\end{document}